\documentclass[11pt,a4paper]{article}

\usepackage[headinclude]{typearea}
\usepackage[english]{babel}   
\usepackage[T1]{fontenc}
\usepackage{scrtime}
\usepackage{amsmath,amsthm,amsfonts,amssymb}
\usepackage {color}  
\usepackage {pifont} 
\usepackage {multicol,multirow} 		
\usepackage[numbers]{natbib}    
\usepackage[normalem]{ulem}               
\usepackage [scaled=.90]{helvet} 
\usepackage {courier} 
\usepackage {graphicx} 
\usepackage {array}   
\usepackage {longtable}  
\usepackage{fancyvrb}
\usepackage{enumerate} 
\usepackage{ifpdf}
\usepackage{caption}
\usepackage{mathrsfs}
\usepackage{setspace}
\usepackage{xcolor}
\usepackage{marvosym}
\usepackage{mathtools}
\mathtoolsset{showonlyrefs} 
\usepackage{verbatim}
\usepackage{dsfont}
\usepackage{hyperref}
\makeatletter
\def\namedlabel#1#2{\begingroup
    #2%
    \def\@currentlabel{#2}%
    \phantomsection\label{#1}\endgroup
}

\newtheorem{theorem}{Theorem}[section]

\newtheorem{lemma}[theorem]{Lemma}
\newtheorem{corollary}[theorem]{Corollary}

\theoremstyle{definition}
\newtheorem{assumption}{Assumption}[section]
\newtheorem{definition}[theorem]{Definition}
\newtheorem{example}[theorem]{Example}
\newtheorem{remark}[theorem]{Remark}

\newcommand{\E}{{\mathbb{E}}}
\newcommand{\N}{{\mathbb{N}}}
\renewcommand{\P}{{\mathbb{P}}}
\newcommand{\Q}{{\mathbb{Q}}}
\newcommand{\R}{{\mathbb{R}}}

\newcommand{\diff}{\mathop{}\!\mathrm{d}}
\renewcommand{\P}{{\mathbb P}}
\newcommand{\D}{{\mathbb{D}}}
\newcommand{\B}{{\mathcal{B}}}
\title{A Skorohod measurable universal functional representation of solutions to semimartingale SDEs}

\author{Pawe{\l} Przyby{\l}owicz\thanks{P. Przyby\l owicz is supported by the National Science Centre, Poland, under project 2017/25/B/ST1/00945.} \and Verena Schwarz \and Alexander Steinicke \and Michaela Sz\"olgyenyi\thanks{V. Schwarz and M. Sz\"olgyenyi are supported by the Austrian Science Fund (FWF): DOC 78.} }

\date{Corrected Version, March 2025}

\begin{document}

\maketitle

\begin{abstract}
In this paper we show the existence of a universal Skorohod measurable functional representation for a large class of semimartingale-driven stochastic differential equations. For this we prove that paths of the strong solutions of stochastic differential equations can be written as measurable functions of the paths of their driving processes into the space of all càdlàg functions equipped with the Borel sigma-field generated by all open sets with respect to the Skorohod metric.
This result can be applied to calculate Malliavin derivatives for SDEs driven by pure-jump Lévy processes with drift.\\

\end{abstract}

\section{Introduction}\label{sec:intro}

In this paper we prove that the universal functional representation for solutions to general semimartingale SDEs is Skorohod measurable.
This result has various applications.
It can be used in Malliavin calculus for pure-jump L\'evy processes, see Theorem \ref{ThmMalCalc} below; there our result yields a universal expression for the Malliavin derivative of an SDE solution, providing access to hedging strategies for price processes of financial assets.
Furthermore, functional representations of SDEs have applications in computational stochastics. In the jump-free case the respective classical result has been applied for example in \cite{muellergronbach2019b, MY20}. In these articles, upper and lower bounds for approximation schemes for Brownian-motion driven SDEs are proven. Our result allows us to perform complexity analysis for jump-diffusion SDEs with irregular coefficients; a first work in this direction is \cite{PSS22}.

The class of SDEs we consider is rather general and therefore covers a wide range of applications.
In particular, we consider SDEs of the form
\begin{equation}
\begin{aligned}\label{SDE}
X_t = H_t + \int_0^t g(s,G,X) \diff Y_s, \quad t\in[0,T]
\end{aligned}
\end{equation}
where $m$, $d$, $r\in\N$, $T\in(0,\infty)$, $Y$ is an $\R^m$-valued càdlàg semimartingale, $H$ is an $\R^d$-valued càdlàg adapted process, and $G$ is an $\R^r$-valued càdlàg adapted process on the filtered probability space $(\Omega,\mathcal{F},(\mathcal{F}_t)_{t\in[0,T]},\P)$ that satisfies the usual conditions. Let for all $n\in\N$, $\D_n = \{f\colon[0,T]\to \R^n: f \text{ is càdlàg}\}$. For the functions $g$ we first define the mapping $f\colon [0,T]\times\D_r\times \D_d \to L(\R^m,\R^d)$. We assume that the mapping $t\mapsto f(t,\zeta,\gamma) $ is càdlàg for all $(\zeta,\gamma)\in\D_r\times\D_d$ and define $g\colon [0,T]\times\D_r\times \D_d \to L(\R^m,\R^d)$ for all $t\in(0,T]$, $(\zeta,\gamma)\in\D_r\times\D_d$ as
\begin{equation}
\begin{aligned}\label{defG}
g(t,\zeta,\gamma) = f(t-,\zeta,\gamma)
\end{aligned}
\end{equation} 
and for all $(\zeta,\gamma)\in\D_r\times\D_d$, $g(0,\zeta,\gamma) = f(0,\zeta,\gamma)$.

The main result in this paper is the existence of a Skorohod measurable function $\Psi$ (only depending on $g$) such that for all $(H,G,Y)$, $X=\Psi(H,G,Y)$. 
This is more general than what could be obtained by the factorization lemma. The factorization lemma, see, e.g.,\cite[Lemma II.11.7]{Bauer2011} and \cite[Corollary 1.97]{klenke2014}, which can also be applied for processes on augmented filtrations, see \cite[Theorem 3.4]{Steinicke2016}, yields that for all $(H,G,Y)$ there exists a function $\Psi=\Psi_{H,G,Y,}$ such that $X=\Psi_{H,G,Y}(H,G,Y)$. In this sense our functional representation is universal.

This universality carries over to the application mentioned above in that the functional representation directly yields a
universal expression for Malliavin derivatives
of solutions to pure-jump L\'evy-driven SDEs.
Before, one could find an expression for the solution's derivative for particular driving processes, which may take quite different forms. Our approach yields a single, unifying expression valid for all pure-jump L\'evy drivers with drift.

A seminal result concerning the functional representation of SDEs is \cite{Yamada1971}, where it is proven that the solution of an SDE driven by time and a Brownian motion can be expressed as a function of the initial value and the Brownian motion. The drawback of this representation is that the function is dependent on the distribution of the solution at time $0$. This has been solved in \cite{Kallenberg1996}, where a universal representation for such SDEs has been proven. In \cite{Karandikar1981,Karandikar1995} and \cite[p.~396, Theorem 12.9]{Karandikar2019} a functional representation for the solution of general semimartingale SDEs has been proven in the same setting as herein, but for the Skorohod function space equipped with the topology of uniform convergence on compact subsets, see \cite[p.~35 f.~and Equation (12.3.10)]{Karandikar2019}. From this result we cannot conclude the Skorohod measuarbility of the functional representation. Although the $\sigma$-field of the Skorohod topology is equal to the $\sigma$-field generated by the balls of the metric of uniform convergence, this is a strict subset of the Borel $\sigma$-field of the topology of uniform convergence, see \cite[p.~157]{Billingsley1999}.\footnote{Separability (which fails for the Borel field of uniform convergence) would guarantee equality of the two.}

In addition, note that our result cannot be concluded from the measurability of $X$, since measurability of a composition of functions does not imply measurability of the functions themselves.

The proof of our result requires major changes in the technique known from \cite{Karandikar2019}. The key novel step in the first part of the proof is showing Skorohod measurability of the approximating sequence; in the second part of the proof, the key lies in proving that the limit of the sequence of stochastic processes, regarded as a sequence of random elements mapping into the space of all càdlàg functions equipped with the Skorohod topology, is indistinguishable from the solution of SDE \eqref{SDE}.

\section{Notations and Preliminaries}\label{sec:pre}

Let for all $n\in\N$, $\D_n = \{f\colon[0,T]\to \R^n: f \text{ is càdlàg}\}$.
We equip this space with the Skorohod topology, which is the topology generated by all coordinate mappings, or equivalently, the topology generated by all pointwise evaluations (cf.~\cite[p.~36]{Karandikar2019}).
For our proof we make use of an alternative definition of the Skorohod topology based on the following metric (cf.~\cite[p.~125 ff.]{Billingsley1999}). 
Define the set 
\begin{equation*}
\begin{aligned}\label{defLambda}
\Lambda = \{\lambda:[0,T]\to[0,T]: \, &\lambda \text{ is strictly increasing and continuous}\\
&\text{with } \lambda(0) = 0 \text{ and } \lambda(T) = T\}
\end{aligned}
\end{equation*}
and then we set for all $\lambda \in\Lambda$,
\begin{equation*}
\begin{aligned}\label{defNormLambda}
\|\lambda\|_\Lambda = \sup_{0\leq s<t\leq T}\Big|\log\Big( \frac{\lambda(t)-\lambda(s)}{t-s}\Big)\Big|.
\end{aligned}
\end{equation*}
Using this we define the Skorohod metric $d^S$ for all $x,y\in\D_n$ as
\begin{equation*}
\begin{aligned}\label{defSkoMetrik}
d^S(x,y) = \inf_{\lambda\in\Lambda} \max \Big\{ \|\lambda\|_\Lambda, \sup_{0\leq t\leq T} \| x(t) - y(\lambda(t))\|\Big\}.
\end{aligned}
\end{equation*}
Note that this metric is equivalent to the one originally used by Skorohod, see \cite[p.~421 f.]{Grikhman1974}.
The space $(\D_n, d^S)$ is separable and complete, see \cite[p.~128, Theorem 12.2]{Billingsley1999}. Moreover, the Skorohod topology is equal to the topology induced by the open subsets of $\D_n$ under $d^S$. The space $\D_n$ with this topology is a Polish space. Further, we denote by $\B(\D_n)$ the Borel $\sigma$-field generated by all open sets in $(\D_n, d^S)$.

Note that the following mappings are Borel measurable
\begin{eqnarray*}
&&\D_n \ni f\to f(t),\\
&&\D_n \ni f\to f(t-),
\end{eqnarray*}
for all  $t\in [0,T]$, cf.~\cite{Pestman1995}.

The space $\D_n$ can also be equipped with the $\sup$-norm $\|\cdot\|_\infty$. By $d$ we denote the corresponding metric. This space is a complete metric space and a topological vector space, however not separable. If we equip $\D_n$ with the Skorohod topology, it becomes separable and hence a Polish space, but then it is not a topological vector space anymore. It follows from the definition by choosing $\lambda$ equal to the identity that for all $x,y\in\D_n$ it holds that $d^S(x,y)\leq d(x,y)$. 
Let $L(\R^m,\R^d)$ be the space of all linear bounded operators from $\R^m$ to $\R^d$  equipped with the operator norm and let $\mathcal{B}(L(\R^m,\R^d))$ be the induced Borel $\sigma$-filed. 

Before we state our assumptions we define for all $n\in\N $, $x\in\D_n$, and $t\in[0,T]$ the function $x^t:[0,T]\to\R^n$, $x^t(s) = x(s\wedge t)$.

\begin{assumption}\label{assf}
For the function $f$ we assume:
\begin{itemize}
    \item[\namedlabel{assf0}{(i)}] the mapping $t\mapsto f(t,\zeta,\gamma) $ is càdlàg for all $(\zeta,\gamma)\in\D_r\times\D_d$,
	\item[\namedlabel{assf1}{(ii)}] $f$ is measurable with respect to $\B([0,T])\otimes\B(\D_r)\otimes\B(\D_d)$,
	\item[\namedlabel{assf2}{(iii)}] $f(t,\zeta,\gamma) = f(t,\zeta^t,\gamma^t)$ for all $\zeta \in\D_r$, $\gamma\in\D_d$, and $t\in[0,T]$,
	\item[\namedlabel{assf3}{(iv)}] there exists a function $C\colon [0,T]\times \D_r \to (0,\infty)$, which is measurable with respect to $\B([0,T])\otimes\B(\D_r)$ and has the property that $t\mapsto C(t,\zeta)$ is càdlàg, such that for all $\zeta\in\D_r$, $\gamma,\gamma_1,\gamma_2\in\D_d$ and $t\in[0,T]$ we have
	\begin{equation*}
	\begin{aligned}
	\| f(t,\zeta,\gamma)\| \leq C(t,\zeta) \Big(1+ \sup_{0\leq s\leq t} \|\gamma(s)\|\Big),
	\end{aligned}
	\end{equation*}
	\begin{equation*}
	\begin{aligned}
	\| f(t,\zeta,\gamma_1) - f(t,\zeta,\gamma_2)\| \leq C(t,\zeta) \Big(\sup_{0\leq s\leq t} \|\gamma_1(s)-\gamma_2(s)\|\Big).
	\end{aligned}
	\end{equation*}  
\end{itemize}
\end{assumption}

\begin{lemma}\label{gmeas}
If $f$ fulfils Assumption \ref{assf} \ref{assf1}, then $g$ defined by \eqref{defG} is $\B([0,T])\otimes\B(\D_r)\otimes\B(\D_d)$-measurable.
\end{lemma}

The proof can be conducted the same way as in the following Lemma which is less intuitive.

\begin{lemma}
Let $f$ satisfy Assumption \ref{assf} and let $g$ be defined by \eqref{defG}. Then for all  $\R^r$-valued càdlàg adapted processes $G$ and all $\R^d$-valued càdlàg adapted processes $X$ it holds that $(g(t,G,X))_{t\in[0,T]}$ is a predictable process.
\end{lemma}

\begin{proof}
    Define the sequence of stochastic processes 
    $(f_n(t,G,X))_{t\in[0,T]}$ for all $n\in\N$, $\omega\in\Omega$ by 
    \[f_n(t,G(\omega),X(\omega)) = f\Big(\frac{\lceil nt\rceil -1}{n},G(\omega),X(\omega)\Big)\mathds{1}_{t\in(0,T]} + f(0,G(\omega),X(\omega))\mathds{1}_{t\in\{0\}}.\]
    This is a sequence of c\`agl\`ad functions that converges to $(g(t,G,X))_{t\in[0,T]}$ as $n\to\infty$ for all $t\in[0,T]$. 
    Since by Assumption \ref{assf} \ref{assf2}, $f_n(s,G(\omega),X(\omega))=f_n(s,G^s(\omega),X^s(\omega))$ for $s=\frac{\lceil nt\rceil-1}{n}$ and since this is $\mathcal{F}_t$-measurable, $f_n$ is adapted. This together with the c\`agl\`ad property of $f_n$ implies that $f_n$ is a generator function of the predictable $\sigma$-algebra. 
    Hence also $g$ is predictable.
\end{proof}

Due to the linear growth Assumption \ref{assf} \ref{assf3} it holds that $(g(t,G,X))_{t\in[0,T]}$ is locally bounded. Since by the previous lemma it is also predictable, we obtain the existence of the stochastic integral $\int_0^t g(s,G,X) \diff Y_s$ for $t\in[0,T].$

\section{Representation of the SDE solution as Skorohod measurable function of the driving processes}\label{sec:main}

The goal of this section is to prove that the solution of SDE \eqref{SDE} can be expressed as a Skorohod measurable function of its initial value, the process $G$, and the semimartingale $Y$. Note that in \cite{Karandikar2019} it is proven that the solution can be expressed as a measurable function with respect to the Borel-$\sigma$-field generated by the topology of uniform convergence. Although the $\sigma$-field of the Skorohod topology is equal to the $\sigma$-field generated by the balls of the metric of uniform convergence, this is a strict subset of the Borel $\sigma$-field of the topology of uniform convergence; separability would guarantee equality of the two. Hence our result cannot be concluded from \cite{Karandikar2019}.

\begin{theorem}\label{Main}
Let Assumption \ref{assf} hold. Then there exists a Skorohod measurable function $\Psi: \D_d \times\D_r\times\D_m \to \D_d$ such that for all $\R^d$-valued càdlàg processes $H$, all $\R^r$-valued càdlàg processes $G$, and all $\R^m$-valued càdlàg semimartingales $Y$, 
\begin{equation*}
\begin{aligned}
X = \Psi(H,G,Y)
\end{aligned}
\end{equation*} 
is the unique solution of 
\begin{equation*}
\begin{aligned}
X_t = H_t + \int_0^t g(s,G,X) \diff Y_s, \quad t\in[0,T].
\end{aligned}
\end{equation*}
\end{theorem}

\begin{remark}
    By \cite[Theorem 12.9]{Karandikar2019} we know that under Assumption \ref{assf} the equation 
    \begin{equation*}
        \begin{aligned}
        X_t = H_t + \int_0^t g(s,G,X) \diff Y_s, \quad t\in[0,T],    
        \end{aligned}
    \end{equation*}
    has a unique solution $(X_t)_{t\in[0,T]}$. Hence in the proof of Theorem \ref{Main} we focus on finding a Skorohod measurable universal functional representation of the solution.
\end{remark}

For the convenience of the reader, we split the proof of this theorem in two parts. In the first part we construct a function $\Psi$ and prove that it is Skorohod measurable. In the second part we show that $X= \Psi(H,G,Y)$ solves the SDE.

For the construction of the function $\Psi$ in the first part we operate similar as in \cite[p.~394]{Karandikar2019}, but with the additional requirement that the function needs to be measurable with respect to the Borel-$\sigma$-algebra generated by the Skorohod topology. Hence, our proof will differ at some crucial points from the one of \cite[p.~394]{Karandikar2019}.

\begin{proof}[Proof of Theorem \ref{Main}, part 1]
We start by defining the mapping
\begin{equation*}
\begin{aligned}\label{M1}
\Psi^{(0)} \colon \D_d\times\D_r\times\D_m \to \D_d,\quad \Psi^{(0)}(\gamma,\zeta,\eta) = \gamma.
\end{aligned}
\end{equation*}
Since for an open set $O\in \B(\D_d)$ it holds that $(\Psi^{(0)})^{-1} (O) = O\times\D_r\times\D_m$ is again an open set, we know that $\Psi^{(0)}$ is continuous and hence also measurable with respect to $(\B(\D_d)\otimes\B(\D_r)\otimes\B(\D_m))/\B(\D_d)$.

Now we define $(\Psi^{(n)})_{n\in\N}$ and prove its measurability inductively:~assume that $\Psi^{(n-1)}\colon \D_d\times\D_r\times\D_m \to \D_d$ is already defined and proven to be measurable. Consider the function 
\begin{equation*}
\begin{aligned}\label{M1a}
&\Gamma^{(n-1)}\colon [0,T]\times\D_d\times\D_r\times\D_m \to L(\R^m,\R^d),\\
&\Gamma^{(n-1)}(t,\gamma,\zeta,\eta) = f(t,\zeta,\Psi^{(n-1)}(\gamma,\zeta,\eta)).
\end{aligned}
\end{equation*}
This function is measurable as a composition of measurable functions. Based on $\Gamma^{(n-1)}$ a sequence of time points $(t_j^{(n)})_{j\in\N}\subset [0,T]$ can be defined iteratively by $t_0^{(n)} = 0$ and 
\begin{equation*}\label{M2}
\begin{aligned}
t_{j+1}^{(n)} &= \inf\big\{s \in[ t_j^{(n)},T] : \| \Gamma^{(n-1)}(s,\gamma,\zeta,\eta)-\Gamma^{(n-1)}(t_j^{(n)},\gamma,\zeta,\eta)\| \geq 2^{-n} \text{ or }\\
&\quad\quad\quad\quad\quad\quad\quad\quad\quad \| \Gamma^{(n-1)}(s-,\gamma,\zeta,\eta)-\Gamma^{(n-1)}(t_j^{(n)},\gamma,\zeta,\eta)\| \geq 2^{-n}\big\} ,
\end{aligned}
\end{equation*}	
where we set $\inf \emptyset = T$. This sequence of time points we write as functions by defining
\[t_{j+1}^{(n)}\colon [0,T]\times \D_d\times\D_r\times\D_m \to [0,T], \quad (t_j^n, \gamma, \zeta, \eta) \mapsto t_{j+1}^{(n)}.\]
This functions are measurable, because it holds for $n$, $j \in \N$ and $c\in[0,T)$ that 
\begin{equation}\label{M3a}
\begin{aligned}
&(t_{j+1}^{(n)})^{-1}([0,c])\\
&= \bigcup_{q\in\Q} \Big(\Big( \bigcup_{s\in[q, c]\cap\Q} \Big( \Big\{(t_j^{(n)}, \gamma, \zeta, \eta) \in [0,T]\times \D_d \times \D_r \times \D_m:\\
&\quad\quad\quad\quad \quad\quad\quad\quad \quad \| \Gamma^{(n-1)}(s,\gamma, \zeta, \eta)-\Gamma^{(n-1)}(t_j^{(n)},\gamma, \zeta, \eta)\| \geq 2^{-n} \Big\}  \\
&\quad\quad\quad\quad \quad\quad\quad\quad  \cap \big([0,q] \times \D_d \times \D_r \times \D_m\big)\Big)\Big)\\
&\quad\quad\quad\quad \cup \Big(\bigcap_{m= 2^{n+1}}^\infty  \bigcup_{s\in [q, c]\cap\Q} \Big( \Big\{(t_j^{(n)}, \gamma, \zeta, \eta) \in [0,T]\times \D_d \times \D_r \times \D_m:\\
&\quad\quad\quad\quad \quad\quad\quad\quad \quad\quad \| \Gamma^{(n-1)}(s,\gamma, \zeta, \eta)-\Gamma^{(n-1)}(t_j^{(n)},\gamma, \zeta, \eta)\| \geq 2^{-n} - m^{-1} \Big\} \\
&\quad\quad\quad\quad \quad\quad\quad\quad \quad\quad\quad\quad \cap \big([0,q] \times \D_d \times \D_r \times \D_m\big) \Big)\Big)\Big).
\end{aligned}
\end{equation}	   
As $\Gamma^{(n-1)}\colon [0,T]\times\D_d\times\D_r\times\D_m \to L(\R^m,\R^d)$ is measurable, $\Gamma^{(n-1)}$ is also measurable for fixed time $s\in[0,T]$ as function from $\D_d\times\D_r\times\D_m$ to $L(\R^m,\R^d)$. As a consequence also the function 
\begin{equation*}\label{M4}
\begin{aligned}
(t_j^{(n)}, \gamma, \zeta, \eta) \mapsto 
 \| \Gamma^{(n-1)}(s,\gamma, \zeta, \eta)-\Gamma^{(n-1)}(t_j^{(n)},\gamma, \zeta, \eta)\| 
\end{aligned}
\end{equation*}	  
is measurable for all $s\in[0,T]$. Hence, $t_{j+1}^{(n)}$ is measurable as its preimage is a countable union of measurable sets, see \eqref{M3a}. 
Recalling $t^{(n)}_0 = 0$, we can prove that $t_{1}^{(n)}$ is measurable as a function from $\D_d\times\D_r\times\D_m$ to $[0,T]$ in a similar way. Now we interpret $t_{j+1}^{(n)}$ as a function from $\D_d\times\D_r\times\D_m$ to $[0,T]$ by plugging in the functions for $t_{j}^{(n)},..., t_{1}^{(n)}$ recursively. This function is measurable as a composition of measurable functions.

As a next step, following \cite[p.~395]{Karandikar2019}, we define $\Psi^{(n)}\colon\D_d\times\D_r\times\D_m \to \D_d$ by 
\begin{equation}\label{M5}
\begin{aligned}
\Psi^{(n)}(\gamma,\zeta,\eta)(s) = \gamma(s) + \sum_{j=0}^\infty \Gamma^{(n-1)}(t_j^{(n)},\gamma,\zeta,\eta)(\eta^{t_{j+1}^{(n)}} (s)-  \eta^{t_{j}^{(n)}} (s)). 
\end{aligned}
\end{equation}
In \cite[p.~395]{Karandikar2019} instead of the interval $[0,T]$ the unbounded interval $[0,\infty)$ is considered. In our case \eqref{M5} simplifies to a finite sum. 
To prove the measurability of $\Psi^{(n)}$,  we  first show that the function $h\colon [0,T] \times \D_m \to \D_m$ defined by $h(t,\eta) = \eta^t$, where as before $\eta^t$ is the process $\eta$ stopped at the time $t$, is measurable. 
For this, we use  \cite[Lemma 6.4.6]{Bogachev2007}, which states that it is sufficient to prove that for fixed $t\in[0,T]$ the mapping $\D_m \ni \eta \mapsto  \eta^t \in\D_m $ is measurable and for fixed $\eta\in\D_m$ the mapping $[0,T] \ni t \mapsto \eta^t\in\D_m$ is right-continuous in $t\in[0,T]$. 

First, we prove that for fixed $t\in[0,T]$ the mapping $\eta \mapsto  \eta^t$ is measurable. For this, we apply \cite[Theorem 4]{Pestman1995}, which claims that it is enough to show that the operator $T_t \colon \D_m \to \D_m$, $T_t\eta = \eta^t$ is a linear bounded operator, if we consider $\D_m$ equipped with the $\sup$-norm instead of the Skorohod topology. The linearity is easy to check and the boundedness follows from the fact that
\begin{equation}
\begin{aligned}\label{M5a}
\|\eta^t \|_\infty = \sup_{s\in[0,T]} \|\eta(t\wedge s) \| = \sup_{s\in[0,t]} \|\eta(s) \| \leq \sup_{s\in[0,T]} \|\eta(s) \| =\|\eta\|_\infty .
\end{aligned}
\end{equation}
Next, we prove that for fixed $\eta\in\D_m$, the mapping $t \mapsto \eta^t$ is right-continuous in $t\in[0,T]$. Let $t\in[0,T]$ and let $(t_k)_{k\in\N}\subset[0,T]$ be a monotone decreasing sequence that converges to $t$. Then with $\lambda = \operatorname{Id}$ we have for all $k\in\N$, 
\begin{equation*}\label{M6}
\begin{aligned}
d^S(\eta^t , \eta^{t_k}) 
\leq \sup_{s\in[0,T]} \|\eta(t\wedge s) - \eta(t_k\wedge s)\| 
= \sup_{s\in[t,t_k]} \|\eta(t) - \eta(s)\|\to 0,\text{ as } k\to\infty  .
\end{aligned}
\end{equation*}
This together with \eqref{M5a} implies the measurability of $h$.

Next, observe that in \eqref{M5}, $\Gamma^{(n-1)}(t_j^{(n)},\gamma, \zeta,\eta)\in L(\R^m,\R^d)$. This can also be considered a bounded linear operator from $\D_m$ to $\D_d$, which is applied for each time point to the càdlàg function evaluated at that point.
Here $\D_m$ and $\D_d$ are equipped with the $\sup$-norm. Then using \cite[p.~387, Theorem 4]{Pestman1995}, we get the measurability of this mapping in the Skorohod topology. Consequently, $\Psi^{(n)}$ is measurable as a composition and sum of measurable functions.

Now we define the function $\Psi\colon\D_d\times\D_r\times\D_m \to\D_d$ by
\begin{equation*}\label{M7}
\begin{aligned}
\Psi(\gamma,\zeta,\eta) = \left\{
\begin{array}{ll}
\lim_{n\to\infty} \Psi^{(n)}(\gamma,\zeta,\eta), & \textrm{if limit exists in Skorohod topology,} \\
0, & \, \textrm{otherwise.} \\
\end{array}
\right. 
\end{aligned}
\end{equation*}
The limit $\lim_{n\to\infty} \Psi^{(n)}(\gamma,\zeta,\eta)$ exists if and only if $(\Psi^{(n)}(\gamma,\zeta,\eta))_{n\in\N}$ is a Cauchy sequence, since $\D_d$ is complete with respect to the Skorohod topology. Hence, the set of all $(\gamma,\zeta,\eta)$ for which the above limit exists is given by
\begin{equation*}\label{M8}
\begin{aligned}
S:=&\Big\{(\gamma,\zeta,\eta)\in \D_d\times\D_r\times\D_m \colon\\
&\quad\quad\quad \forall \varepsilon >0\, \exists\,  n_0 \in\N : \forall\, m,n\geq n_0\colon d^S(\Psi^{(n)}(\gamma,\zeta,\eta),\Psi^{(m)}(\gamma,\zeta,\eta))<\varepsilon\Big\}\\
=& \bigcap_{k\in\N} \bigcup_{n_0\in\N} \bigcap_{m,n\geq n_0} \Big\{(\gamma,\zeta,\eta)\in \D_d\times\D_r\times\D_m \colon\\
&\quad\quad\quad\quad \quad\quad\quad\quad \quad d^S(\Psi^{(n)}(\gamma,\zeta,\eta),\Psi^{(m)}(\gamma,\zeta,\eta))<\frac{1}{k}\Big\}.
\end{aligned}
\end{equation*}
Hence, $\Psi^{(n)}\mathds{1}_S$ is measurable, because $S$ is measurable and $\Psi =\lim_{n\to\infty} \Psi^{(n)}\mathds{1}_S$ is measurable as the limit of measurable functions. This closes part $1$ of the proof of Theorem \ref{Main}.
\end{proof}

To prove that the above constructed Skorohod measurable function is indeed the function we are looking for in Theorem \ref{Main}, we need the following definitions and notations. In order to ensure uniform convergence on $[0,T]$, we consider in the following the time horizon $T +1$, but note that the definitions are of course valid for arbitrary time interval $[0, \widetilde T]$ with $\widetilde T \in (0,\infty)$.

\begin{definition}[cf.~\text{\cite[p.~368, Definition~11.4]{Karandikar2019}}]
A càdlàg adapted increasing process $V$ is said to be a dominating process for a real-valued semimartingale $\xi$ on $[0,T+1]$, if there exists a decomposition $\xi=M+A$, with a càdlàg locally square integrable martingale  $M$ with $M_0 = 0$, and a càdlàg process $A$ with finite variation paths such that the process $B$ defined by 
\begin{equation*}\label{DefDomPro}
\begin{aligned}
	B_t = V_t - 2\sqrt{2} ([M,M]_t +\langle M,M\rangle_t)^{1/2} -\sqrt{2}|A|_t, \quad t\in[0,T+1]
\end{aligned}
\end{equation*} 
is an increasing process with $B_0\geq0$.
\end{definition}

\begin{remark}\label{ExDomPr}
It is known that for each semimartingale $\xi$ there exists a dominating process $V$, see \cite[p.~368, Theorem~11.5]{Karandikar2019}. Furthermore, for each stopping time $\nu\colon \Omega\to[0,T+1]$ it holds that
\begin{equation*}\label{PropDomPro}
\begin{aligned}
	\E[\sup_{0\leq t<\nu}|\xi_t|^2] \leq \E[V_{\nu-}^2],
\end{aligned}
\end{equation*} 
see \cite[p.~369, Theorem~11.7]{Karandikar2019}.
\end{remark}

We also need the following theorem.

\begin{theorem}[cf.~\text{\cite[p.~373, Theorem~11.13]{Karandikar2019}}]\label{ThmEstTheta}
Let $\xi$ be a semimartingale and let $P$ be a locally bounded predictable process. Let $V$ be a dominating process for $\xi$. Then for any stopping time $\nu\colon \Omega\to[0,T+1]$ it holds
\begin{equation*}\label{EstTheta}
\begin{aligned}
	\E\Big[\sup_{0\leq t<\nu} \Big|\int_0^t P_s \diff \xi_s \Big|^2\Big] \leq \E[\theta_{\nu-}^2(P,V)],
\end{aligned}
\end{equation*}
where $\theta_t(P,V)$ is for all $t\in[0,T+1]$ defined by
\begin{equation*}\label{DefTheta}
\begin{aligned}
	\theta_t(P,V) = \Big(\int_0^t |P_s|^2\diff V_s^2\Big)^{1/2}+\int_0^t |P_s|\diff V_s.
\end{aligned}
\end{equation*}
Further,
\begin{equation*}\label{EstTheta2}
\begin{aligned}
	\E\Big[\sup_{0\leq t<\nu} \Big|\int_0^t P_s \diff \xi_s \Big|^2\Big] \leq 4 \E\Big[\Big( \sup_{0\leq s<\nu} |P_s^2|\Big) \textcolor{blue}{V_{\nu-}^2}\Big].
\end{aligned}
\end{equation*}
\end{theorem}

We continue with the second part of the proof of Theorem \ref{Main}, where we show that $\Psi$ constructed in part 1 of the proof has the property that $X= \Psi(H,G,Y)$ is the solution of SDE \eqref{SDE}.

\begin{proof}[Proof of Theorem \ref{Main}, part 2]
To prove the desired property we proceed similar as in the first part of the proof, but instead of defining a sequence of functions approximating $\Psi$, we define a sequence of càdlàg adapted stochastic processes approximating the solution of SDE \eqref{SDE}. First, we set $Z^{(0)} = H$. Next, we define $Z^{(n)}$ inductively as follows:~assume that $Z^{(0)},...,Z^{(n-1)}$ are already defined. We define a sequence $(\tau_j^{(n)})_{j\in\N}$ by $\tau_0^{(n)} = 0$ and for $j\geq 1$ recursively by
\begin{equation*}\label{M10}
\begin{aligned}
\tau_{j+1}^{(n)} &= \inf\Big(\big\{s \in[ \tau_j^{(n)},T] : \| f(s,G,Z^{(n-1)})-f(\tau_j^{(n)},G,Z^{(n-1)})\| \geq 2^{-n} \text{ or }\\
&\quad\quad\quad\quad \quad\quad\quad\quad  \| f(s-,G,Z^{(n-1)})-f(\tau_j^{(n)},G,Z^{(n-1)})\| \geq 2^{-n}\big\} \cup \{T\} \Big).
\end{aligned}
\end{equation*}
Since $s\mapsto  \| f(s,G,Z^{(n-1)})-f(\tau_j^{(n)},G,Z^{(n-1)})\|$ is a càdlàg adapted process, this is a sequence of stopping times. Further, it holds that $\lim_{j\to\infty} \tau^{(n)}_j = T$, in particular $(\tau_j^{(n)})_{j\in\N}$ is $\P$-a.s.~eventually constant $T$. Now we are able to define $Z^{(n)}_0 = H_0 $ and for all $j\in\N_0$, $\tau^{(n)}_j < t \leq \tau^{(n)}_{j+1}$ by 
\begin{equation*}\label{M11}
\begin{aligned}
Z^{(n)}_t = H_t + f(\tau^{(n)}_j, G, Z^{(n-1)})(Y_t - Y_{\tau^{(n)}_j}).
\end{aligned}
\end{equation*}
Equivalent for all $t\in[0,T]$,
\begin{equation}\label{M12}
\begin{aligned}
Z^{(n)}_t = H_t + \sum_{j=0}^{\infty} f(\tau^{(n)}_j, G, Z^{(n-1)})\Big(Y^{\tau^{(n)}_{j+1}}_{t}  - Y^{\tau^{(n)}_{j}}_{t}\Big).
\end{aligned}
\end{equation}
Comparing equation \eqref{M12} and \eqref{M5} we obtain that for almost all $\omega\in\Omega$ it holds
\begin{equation*}\label{M12a}
\begin{aligned}
Z^{(n)}_t(\omega) = \Psi^{(n)}( H(\omega), G(\omega), Y(\omega)).
\end{aligned}
\end{equation*}
In the next step we explore the convergence of the sequence $(Z^{(n)})_{n\in\N}$. For this we consider in the following the time interval $[0,T+1]$. We extend the filtration $(\mathcal{F}_t)_{t\in[0,T]}$ to $(\mathcal{F}_t)_{t\in[0,T+1]}$ by defining $\mathcal{F}_t = \mathcal{F}_T$ for all $t\in(T,T+1]$. Further we use for all $n\in\N$ the notation $\D_n([0,T+1]) = \{f\colon[0,T+1]\to \R^n: f \text{ is càdlàg}\}$ and define $\widetilde f\colon [0,T+1]\times\D_r([0,T+1])\times \D_d([0,T+1]) \to L(\R^m,\R^d)$ by $\widetilde f(t,\zeta,\gamma) = f(t,\zeta|_{[0,T]},\gamma|_{[0,T]}) $ for $t\in[0,T]$ and $\widetilde f(t,\zeta,\gamma) = f(T,\zeta|_{[0,T]},\gamma|_{[0,T]}) $ for $t\in(T,T+1]$.
Moreover, we define $H, G,Y,(Z^{(n)})_{n\in\N}$ on the interval $[0,T+1]$ by extending them constantly on the interval $[T,T+1]$.

We follow \cite[p.~396]{Karandikar2019} and define for all $n\in\N$ the sequences $(S^{(n)})_{n\in\N}$ and $(R^{(n)})_{n\in\N}$ through
\begin{equation}\label{M13}
\begin{aligned}
S^{(n)}_t &= \sum_{j=0}^\infty \widetilde f(\tau_j^{(n)},G, Z^{(n-1)})\mathds{1}_{[\tau^{(n)}_j, \tau^{(n)}_{j+1})}(t) +\widetilde f(T,G, Z^{(n-1)})\mathds{1}_{[T,T+1]}(t), \quad t\in[0,T+1],\\
R^{(n)}_t &= H_t + \int_0^t \widetilde f(s-, G,Z^{(n-1)})\diff Y_s, \quad t\in[0,T+1].
\end{aligned}
\end{equation}
It holds that 
\begin{equation}\label{M14}
\begin{aligned}
Z^{(n)}_t = H_t + \int_0^t S_{s-}^{(n)}\diff Y_s,
\end{aligned}
\end{equation}
here we need $s-$ in the integrand, because the half open interval $[\tau_j^{(n)}, \tau_{j+1}^{(n)})$ needs to have the form $(\tau_j^{(n)}, \tau_{j+1}^{(n)}]$ to deliver the correct integral. Furthermore,  
\begin{equation}\label{M15}
\begin{aligned}
\big\|S_t^{(n)} -\widetilde f(t,G,Z^{(n-1)})\big\| \leq 2^{-n}.
\end{aligned}
\end{equation}
Since $Y$ is an $\R^m$-valued semimartingale, we know that each component of $Y$ is a semimartingale and hence admits a dominating process by Remark \ref{ExDomPr}. Summing up these dominating processes we obtain a common dominating process, which we denote by $V$. Now we define the process $U$ by
\begin{equation*}\label{M16}
\begin{aligned}
U_t = V_t+V_t^2 + \sup_{0\leq s \leq t} \|H_s\| + \widetilde C(t,G),\quad t\in[0,T+1].
\end{aligned}
\end{equation*} 
Here $\widetilde C(t,G(\omega))$ is defined using $C(t,G|_{[0,T]}(\omega))$  as in Assumption \ref{assf} \ref{assf3} for $t\in[0,T]$ by $\widetilde C(t,\zeta) = C(t,\zeta|_{[0,T]})$ and for $t\in(T,T+1]$ by $\widetilde C(t,\zeta) = C(T,\zeta|_{[0,T]})$. Further define the sequence of stopping times $(\nu_j)_{j\in\N}$ for all $j\in\N$ by 
\begin{equation}\label{M17}
\begin{aligned}
\nu_j = \inf\big(\{t\in[0,T\textcolor{blue}{+1}]: U_t \geq j \text{ or } U_{t-}\geq j\}\cup\{T+1\}\big).
\end{aligned}
\end{equation} 
The sequence defined in \eqref{M17} is indeed a sequence of stopping times, since $U$ is a càdlàg process, see \cite[Theorem 2.46]{Karandikar2019}.
Further, it holds that $\lim_{j\to\infty} \nu_j = T+1$. 
Next, combining \eqref{M13} and \eqref{M14} we get 
\begin{equation}\label{M18}
\begin{aligned}
&\E\Big[\sup_{0\leq t< \nu_j}\|R^{(n)}_t - Z^{(n)}_t\|^2\Big]\\
&= \E\Big[\sup_{0\leq t< \nu_j}\Big\|H_t +\int_0^t \widetilde f(s-,G, Z^{(n-1)})\diff Y_s - H_t - \int_0^t S_{s-}^{(n)}\diff Y_s\Big\|^2\Big]\\
&= \E\Big[\sup_{0\leq t< \nu_j}\Big\|\int_0^t \big(\widetilde f(s-,G, Z^{(n-1)})- S_{s-}^{(n)}\big)\diff Y_s\Big\|^2\Big].
\end{aligned}
\end{equation} 
To continue this calculation, recall the following standard notation. For a matrix $A$ we denote by $A_{ij}$ the entry of $A$ in the $i$-th row and the $j$-th column and for a vector $b$ we denote by $b_j$ the $j$-th entry of the vector. Hence, for $i\in\{1,..,d\}$ the $i$-th entry of the vector
\[\int_0^t (\widetilde f(s-,G, Z^{(n-1)})- S_{s-}^{(n)})\diff Y_s\]
is given by 
\begin{equation}\label{M19}
\begin{aligned}
&\Big(\int_0^t \big(\widetilde f(s-,G, Z^{(n-1)})- S_{s-}^{(n)}\big)\diff Y_s\Big)_i\\
&\quad\quad = \sum_{k=1}^{m} \int_0^t \big(\widetilde f(s-,G, Z^{(n-1)})- S_{s-}^{(n)}\big)_{ik}\diff (Y_s)_k.
\end{aligned}
\end{equation} 
Using \eqref{M19} and applying the Cauchy-Schwarz inequality to \eqref{M18} we obtain that
\begin{equation}\label{M20}
\begin{aligned}
&\E\Big[\sup_{0\leq t< \nu_j}\Big\|\int_0^t \big(\widetilde f(s-,G, Z^{(n-1)})- S_{s-}^{(n)}\big)\diff Y_s\Big\|^2\Big]\\
&\leq \E\Big[\sup_{0\leq t< \nu_j}\sum_{i=1}^{d} m \sum_{k=1}^m \Big|\int_0^t \big(\widetilde f(s-,G, Z^{(n-1)})- S_{s-}^{(n)}\big)_{ik}\diff (Y_s)_k\Big|^2\Big]\\
&\leq \sum_{i=1}^{d} m \sum_{k=1}^m \E\Big[ \sup_{0\leq t< \nu_j}\Big|\int_0^t \big(\widetilde f(s-,G, Z^{(n-1)})- S_{s-}^{(n)}\big)_{ik}\diff (Y_s)_k\Big|^2\Big].\\
\end{aligned}
\end{equation}
We combine this with the second statement of Theorem \ref{ThmEstTheta} and use the definition of the stopping time $\nu_j$ in \eqref{M17} as well as the estimate in \eqref{M15} to obtain that
\begin{equation}\label{M21}
\begin{aligned}
&\sum_{i=1}^{d} m \sum_{k=1}^m \E\Big[ \sup_{0\leq t< \nu_j}\Big|\int_0^t (\widetilde f(s-,G, Z^{(n-1)})- S_{s-}^{(n)})_{ik}\diff (Y_s)_k\Big|^2\Big]\\
&\leq 4 \sum_{i=1}^{d} m \sum_{k=1}^m \E\Big[ \sup_{0\leq s< \nu_j}\Big|(\widetilde f(s-,G, Z^{(n-1)})- S_{s-}^{(n)})_{ik}\Big|^2 V_{\nu_j}^2\Big]\\
&\leq 4 \sum_{i=1}^{d} m \sum_{k=1}^m \E\Big[ \sup_{0\leq s< \nu_j}\Big|(\widetilde f(s-,G, Z^{(n-1)})- S_{s-}^{(n)})_{ik}\Big|^2 j^2\Big]\\
&\leq 4 \sum_{i=1}^{d} m \sum_{k=1}^m \cdot 2^{-2n} j^2 = 4 d m^2\, 2^{-2n}\, j^2.
\end{aligned}
\end{equation}
Combining \eqref{M18}, \eqref{M20}, and \eqref{M21} we conclude that 
\begin{equation*}\label{M22}
\begin{aligned}
&\E\Big[\sup_{0\leq t< \nu_j}\|R^{(n)}_t - Z^{(n)}_t\|^2\Big] \leq  4 d m^2\, 2^{-2n}\, j^2.
\end{aligned}
\end{equation*}
Next, we define for all $n\in\N_0$, $t\in[0,T+1]$,
\begin{equation*}\label{M23}
\begin{aligned}
A^{(n)}_t = \sup_{0\leq s\leq t} \|Z_s^{(n+1)} -Z_s^{(n)}\|.
\end{aligned}
\end{equation*}
For any  stopping time $ \tau\leq\nu_j$ and for all $n\in\N$ we estimate the second moment of $A^{(n)}_{\tau-}$ by
\begin{equation}\label{M24}
\begin{aligned}
&\E\Big[\big(A^{(n)}_{\tau-}\big)^2\Big]
=\E\Big[\sup_{0\leq t < \tau} \|Z_t^{(n+1)} -Z_t^{(n)}\|^2\Big]\\
&\leq 3\E\Big[\sup_{0\leq t< \tau} \|R_t^{(n+1)} -Z_t^{(n+1)}\|^2\Big] +3\E \Big[\sup_{0\leq t< \tau} \|R_t^{(n)} -Z_t^{(n)}\|^2\Big]\\
&\quad +3\E \Big[\sup_{0\leq t< \tau} \|R_t^{(n+1)} -R_t^{(n)}\|^2\Big]\\
&\leq 12 d m^2 j^2 2^{-2n} (1+ 2^{-2}) + 3\E \Big[\sup_{0\leq t< \tau} \|R_t^{(n+1)} -R_t^{(n)}\|^2\Big].
\end{aligned}
\end{equation}
Further, 
\begin{equation}\label{M25}
\begin{aligned}
&\E \Big[\sup_{0\leq t< \tau} \|R_t^{(n+1)} -R_t^{(n)}\|^2\Big]\\
&= \E \Big[\sup_{0\leq t<\tau} \Big\|\int_0^t \big(\widetilde f(s-, G,Z^{(n)})- \widetilde f(s-, G,Z^{(n-1)})\big)\diff Y_s\Big\|^2\Big]\\
&\leq \E\Big[\sup_{0\leq t< \tau}\sum_{i=1}^{d} m \sum_{k=1}^m \Big|\int_0^t \big(\widetilde f(s-,G, Z^{(n)})- \widetilde f(s-, G,Z^{(n-1)})\big)_{ik}\diff (Y_s)_k\Big|^2\Big]\\
&\leq \sum_{i=1}^{d} m \sum_{k=1}^m \E\Big[ \sup_{0\leq t< \tau}\Big|\int_0^t \big(\widetilde f(s-,G, Z^{(n)})- \widetilde f(s-, G,Z^{(n-1)})\big)_{ik}\diff (Y_s)_k\Big|^2\Big].\\
\end{aligned}
\end{equation}
Next, we apply the first statement of Theorem \ref{ThmEstTheta},  to estimate for all $i \in\{1,..,d\}$, $k\in\{1,...,m\}$,
\begin{equation*}\label{M26}
\begin{aligned}
& \E\Big[ \sup_{0\leq t< \tau}\Big|\int_0^t \big(\widetilde f(s-,G, Z^{(n)})- \widetilde f(s-, G,Z^{(n-1)})\big)_{ik}\diff (Y_s)_k\Big|^2\Big]\\
&\leq \E \big[\theta^2_{\tau-} \big(\big(\widetilde f(s-,G, Z^{(n)})- \widetilde f(s-, G,Z^{(n-1)})\big)_{ik},V\big)\big]\\
&\leq \E \Big[ \Big(\Big[\int_0^{\tau-} \big|\big(\widetilde f(s-,G, Z^{(n)})- \widetilde f(s-, G,Z^{(n-1)})\big)_{ik}\big|^2 \diff V_s^2\Big]^{1/2} \\
&\quad\quad\quad\quad + \int_0^{\tau-} \big|\big(\widetilde f(s-,G, Z^{(n)})- \widetilde f(s-, G,Z^{(n-1)})\big)_{ik}\big| \diff V_s\Big)^2\Big]\\
&\leq \E \Big[ \Big(\Big[\int_0^{\tau-} \widetilde C(s-,G)^2 \sup_{0\leq u< s} \|Z^{(n)}_u - Z^{(n-1)}_u\|^2 \diff V_s^2\Big]^{1/2} \\
&\quad\quad\quad\quad + \int_0^{\tau-} \widetilde C(s-,G) \sup_{0\leq u< s} \|Z^{(n)}_u - Z^{(n-1)}_u\| \diff V_s\Big)^2\Big]\\
&\leq j^2 \E \Big[ \Big(\Big[\int_0^{\tau-} \sup_{0\leq u< s} \|Z^{(n)}_u - Z^{(n-1)}_u\|^2 \diff V_s^2\Big]^{1/2} \\
&\quad\quad\quad\quad + \int_0^{\tau-} \sup_{0\leq u< s} \|Z^{(n)}_u - Z^{(n-1)}_u\| \diff V_s\Big)^2\Big]\\
&= j^2 \E \Big[ \Big(\Big[\int_0^{\tau-} \big(A_{s-}^{(n-1)}\big)^2 \diff V_s^2\Big]^{1/2}
+ \int_0^{\tau-} A_{s-}^{(n-1)} \diff V_s\Big)^2\Big]\\
&\leq 2\, j^2 \E \Big[ \int_0^{\tau-} \big(A_{s-}^{(n-1)}\big)^2 \diff V_s^2
+ \Big(\int_0^{\tau-} A_{s-}^{(n-1)} \diff V_s\Big)^2\Big].
\end{aligned}
\end{equation*}
Here we used additionally Assumption \ref{assf} \ref{assf3}, the definition of $\nu_j$, and the definition of $A_{t-}^{(n-1)}$. In the next step we apply Jensen's inequality and obtain
\begin{equation}\label{M26a}
\begin{aligned}
& \E\Big[ \sup_{0\leq t< \tau}\Big|\int_0^t \big(\widetilde f(s-,G, Z^{(n)})- \widetilde f(s-, G,Z^{(n-1)})\big)_{ik}\diff (Y_s)_k\Big|^2\Big]\\
&\leq 2\, j^2 \E \Big[ \int_0^{\tau-} \big(A_{s-}^{(n-1)}\big)^2 \diff V_s^2
+ V_{\tau-} \int_0^{\tau-} \big(A_{s-}^{(n-1)}\big)^2 \diff V_s\Big]\\
&\leq 2\, j^3 \E \Big[ \int_0^{\tau-} \big(A_{s-}^{(n-1)}\big)^2 \diff ( V_s^2 +V_s) \Big].
\end{aligned}
\end{equation}
 Plugging \eqref{M26a} and  \eqref{M25} into \eqref{M24} yields
\begin{equation}\label{M27}
\begin{aligned}
&\E\big[\big(A^{(n)}_{\tau-}\big)^2\big]\\
&\leq 15 d m^2 j^2 2^{-2n} + 3\sum_{i=1}^{d} m \sum_{k=1}^m 2\, j^3 \E \Big[ \int_0^{\tau-} \big(A_{s-}^{(n-1)}\big)^2 \diff ( V_s^2 +V_s) \Big]\\
&= 15 d m^2 j^2 2^{-2n} + 6 d m^2 j^3 \E \Big[ \int_0^{\tau-} \big(A_{s-}^{(n-1)}\big)^2 \diff ( V_s^2 +V_s) \Big].
\end{aligned}
\end{equation}
Using similar considerations as in \eqref{M20} and the second statement of Theorem \ref{ThmEstTheta} we get
\begin{equation}\label{M28}
\begin{aligned}
&\E\big[\big(A_{\tau-}^{(0)}\big)^2\big] = \E\Big[\Big(\sup_{0\leq t <\tau} \|Z_t^{(1)} -Z_t^{(0)}\|\Big)^2\Big]\\
&=\E\Big[\Big(\sup_{0\leq t<\tau} \Big\|H_t+\int_0^t S_{s-}^{(1)} \diff Y_s-H_t\Big\|\Big)^2\Big]\\
&\leq \E\Big[\sup_{0\leq t< \tau}\sum_{i=1}^{d} m \sum_{k=1}^m \Big|\int_0^t (S_{s-}^{(1)})_{ik}\diff (Y_s)_k\Big|^2\Big]\\
&\leq \sum_{i=1}^{d} m \sum_{k=1}^m \E\Big[ \sup_{0\leq t< \tau}\Big|\int_0^t (S_{s-}^{(1)})_{ik}\diff (Y_s)_k\Big|^2\Big]\\
&\leq 4 \sum_{i=1}^{d} m \sum_{k=1}^m   \E\Big[ \Big(\sup_{0\leq t< \tau} |(S_{t-}^{(1)})_{ik}|^2  V_{\tau-}^2 \Big)\Big]\\
&\leq  4\,dm^2  \E\Big[\Big( \sup_{0\leq t< \tau} \|S_{t-}^{(1)}\|^2 \Big) V_{\tau-}^2 \Big]\\
&\leq  4\,dm^2 j^2 \E\Big[ \sup_{0\leq t< \tau} \|S_{t-}^{(1)}\|^2 \Big].
\end{aligned}
\end{equation}
Using \eqref{M13} and Assumption \ref{assf} \ref{assf3} we obtain that
\begin{equation}\label{M28a}
\begin{aligned}
&\E\Big[ \sup_{0\leq t< \tau} \|S_{t-}^{(1)}\|^2 \Big]\\
&\leq \E\Big[ \sup_{0\leq t< \tau} \Big\|\sum_{j=0}^\infty \widetilde f(\tau_j^{(1)},G, Z^{(0)})\mathds{1}_{(\tau^{(1)}_j, \tau^{(1)}_{j+1}]}(t) \Big\|^2 \Big]\\
&= \E\Big[ \sup_{0\leq t< \tau} \sum_{j=0}^\infty \Big\| \widetilde f(\tau_j^{(1)},G, H)\Big\|^2 \mathds{1}_{(\tau^{(1)}_j, \tau^{(1)}_{j+1}]}(t) \Big]\\
&\leq \E\Big[ \sup_{0\leq t< \tau} \sum_{j=0}^\infty \| \widetilde C(\tau_j^{(1)} ,G) \|^2 \Big(1+ \sup_{0\leq s \leq \tau_j^{(1)}} \|H_s\|\Big)^2 \mathds{1}_{(\tau^{(1)}_j, \tau^{(1)}_{j+1}]}(t) \Big]\\
&\leq\E\Big[ \sup_{0\leq t< \tau} \| \widetilde C(t ,G) \|^2 \sup_{0\leq t< \tau} (1+  \|H_t\|)^2 \Big]\\
&\leq j^2 (1+j)^2.
\end{aligned}
\end{equation}
Plugging \eqref{M28a} into \eqref{M28} we get
\begin{equation}\label{M28b}
\begin{aligned}
&\E\big[\big(A_{\tau-}^{(0)}\big)^2\big] \leq 
 4 d m^2 j^4 (1+j)^2.
\end{aligned}
\end{equation}
Next, we define for all $t\in[0,T+1]$,
\begin{equation*}\label{M29}
\begin{aligned}
B_t = \sum_{n=0}^\infty 2^n (A_t^{(n)})^2.
\end{aligned}
\end{equation*}
For some stopping time $\tau\leq \nu_j$, \eqref{M27} and \eqref{M28b} assure
\begin{equation*}\label{M30}
\begin{aligned}
&\E\big[B_{\tau-}\big] = \E\Big[\sum_{n=0}^\infty 2^n \big(A_{\tau-}^{(n)}\big)^2\Big]
= \sum_{n=0}^\infty 2^n \E\big[\big(A_{\tau-}^{(n)}\big)^2\big]\\
&\leq \E\big[\big(A_{\tau-}^{(0)}\big)^2\big]\\
&\quad + \sum_{n=1}^\infty\Big( 2^n 15 d m^2 j^2 2^{-2n} +  2^n  6 d m^2 j^3 \E \Big[ \int_0^{\tau-} \big(A_{s-}^{(n-1)}\big)^2 \diff ( V_s^2 +V_s) \Big]\Big)\\
&\leq 4 d m^2 j^4 (1+j)^2 + 15 d m^2 j^2 \sum_{n=1}^\infty 2^{-n} \\
&\quad + 12 d m^2 j^3 \E \Big[ \sum_{n=1}^\infty 2^{n-1} \int_0^{\tau-} \big(A_{s-}^{(n-1)}\big)^2 \diff ( V_s^2 +V_s) \Big]\\
&= 4 d m^2 j^4 (1+j)^2 +  15 d m^2 j^2 +   12 d m^2 j^3 \E \Big[\int_0^{\tau-} B_{s-} \diff (V_s^2 + V_s)\Big].
\end{aligned}
\end{equation*}
Now we apply Gronwall's inequality, see \cite[Theorem~12.1]{Karandikar2019}, to conclude that there exists a constant $c\in(0,\infty)$, which only depends on $j$, such that
\begin{equation*}\label{M31}
\begin{aligned}
\E[B_{\nu_j -}] \leq c.
\end{aligned}
\end{equation*}
Hence, for all $j\in\N$ we have proven that 
\[\sum_{n=0}^\infty 2^n \E\big[\big(A_{\nu_j-}^{(n)}\big)^2 \big]<\infty.\]
This implies that for large $n\in\N$, $\E\big[\big(A_{\nu_j -}^{(n)}\big)^2\big]< 2^{-n}$. Therefore, it holds that 
\begin{equation*}\label{M32}
\begin{aligned}
&\sum_{n=0}^\infty \Big(\E\big[\big(A_{\nu_j -}^{(n)}\big)^2\big]\Big)^{\frac{1}{2}} = \sum_{n=0}^\infty  \Big(\E\Big[\sup_{0\leq s<\nu_j} \| Z^{(n+1)}_s - Z^{(n)}_s\|^2 \Big]\Big)^{\frac{1}{2}}\\
&= \sum_{n=0}^\infty  \Big\|\sup_{0\leq s<\nu_j} \| Z^{(n+1)}_s - Z^{(n)}_s\|\Big\|_2 <\infty.
\end{aligned}
\end{equation*}
Hence, we obtain 
\begin{equation}\label{M33}
\begin{aligned}
\Big\| \sum_{n=0}^\infty \sup_{0\leq s <\nu_j} \|Z^{(n+1)}_s - Z^{(n)}_s\| \Big\|_2 <\infty.
\end{aligned}
\end{equation}
This implies that
\begin{equation*}\label{M35}
\begin{aligned}
N= \bigcup_{j=1}^\infty \Big\{\omega \in\Omega : \sum_{n=0}^\infty \sup_{0\leq s<\nu_j} \|Z^{(n+1)}_s - Z^{(n)}_s\| = \infty \Big\}
\end{aligned}
\end{equation*}
is a $\P$-null set. Hence, for all $\omega \in\Omega\setminus N$ the sequence $(Z_s^{(n)}(\omega))_{n\in\N}$ converges uniformly on $[0,\nu_j (\omega))$ for all $j\in\N$.Since $\P(N)=0$ and because of the definition of $\nu_j$ we have that $Z^{(n)}$ converges uniformly on $[0,T+1)$. 
Therefore, the process $(\widetilde Z_t)_{t\in[0,T]}$ defined by
\begin{equation*}\label{M36}
\begin{aligned}
\widetilde Z_t(\omega) = \left\{
\begin{array}{ll}
\lim_{n\to\infty} Z^{(n)}_t(\omega), & \omega  \in\Omega\setminus N, \\
0, & \omega \in N \\
\end{array}
\right.
\end{aligned}
\end{equation*}
is well-defined and $Z^{(n)}$ converges uniformly on $[0,T]$ to $\widetilde Z$ almost surely. Hence, $Z^{(n)}$ also converges almost surely to $\widetilde Z$ in the Skorohod metric, since by the definition of the Skorohod metric it follows that every convergent sequence in the $\sup$-norm is also convergent in the Skorohod metric.
Next we define a new process for all $t\in[0,T]$ by
\begin{equation*}\label{M37}
\begin{aligned}
Z_t(\omega) = \left\{
\begin{array}{ll}
\lim_{n\to\infty} Z^{(n)}_t(\omega), & \text{if the limit exists in the Skorohod topology,}\\
0, & \text{otherwise.} \\
\end{array}
\right.
\end{aligned}
\end{equation*}
We observe that $Z$ and $\widetilde Z$ differ only on the subset of a nullset and are hence indistinguishable. Hence, $Z^{(n)} = \Psi^{(n)}(H,G,Y)$ converges almost surely to $Z$ in the Skorohod topology. Hence, $Z = \Psi(H,G, Y)$ a.s.

Furthermore, it follows from \eqref{M33} that
\begin{equation}\label{M34}
\begin{aligned}
&\sup_{k\geq1} \Big\|\sup_{0\leq s<\nu_j} \|Z^{(n+k)}_s -Z^{(n)}_s\| \Big\|_2 
\leq \sup_{k\geq1} \Big\|\sum_{j=n+1}^{n+k} \sup_{0\leq s<\nu_j} \|Z^{(j)}_s -Z^{(j-1)}_s\| \Big\|_2\\ 
&\leq \sum_{j=n+1}^{\infty} \Big\| \sup_{0\leq s<\nu_j} \|Z^{(j)}_s -Z^{(j-1)}_s\| \Big\|_2 \to 0 \text{ as } n\to\infty.
\end{aligned}
\end{equation}
Extending $\widetilde Z$ to the interval $[0,T+1]$ by defining $\widetilde Z_t =\widetilde Z_T$ and using \eqref{M34} we obtain
\begin{equation}\label{M38}
\begin{aligned}
\lim_{n\to\infty}\E\Big[\sup_{0\leq s<\nu_j} \|\widetilde Z_s - Z_s^{(n)}\|^2\Big] =0.
\end{aligned}
\end{equation}
Further, we get by Assumption \ref{assf} \ref{assf3} and the definitions of $U$ and $\nu_j$ that
\begin{equation}\label{M39}
\begin{aligned}
\sup_{0\leq s<\nu_j}  \| \widetilde f(s,G,\widetilde Z)- \widetilde f(s,G, Z^{(n)})\|
&\leq \sup_{0\leq s<\nu_j} \widetilde C(s,G) \|\widetilde Z - Z^{(n)}\|\\
&\leq j \sup_{0\leq s<\nu_j}  \|\widetilde Z_s - Z_s^{(n)}\|.
\end{aligned}
\end{equation}
The second statement of Theorem \ref{ThmEstTheta} together with \eqref{M39} and \eqref{M38} yields
\begin{equation*}\label{M40}
\begin{aligned}
&\E\Big[\sup_{0\leq s < \nu_j}  \Big\| H_t + \int_0^t \widetilde f(s,G,\widetilde Z) \diff Y_s - H_t - \int_0^t \widetilde f(s,G, Z^{(n)}) \diff Y_s \Big\|^2\Big]\\
&=\E\Big[\sup_{0\leq s< \nu_j}  \Big\| \int_0^t \big(\widetilde f(s,G,\widetilde Z) - \widetilde f(s,G, Z^{(n)})\big) \diff Y_s \Big\|^2\Big]\\
&\leq \E\Big[\sup_{0\leq s< \nu_j}\sum_{i=1}^{d} m \sum_{k=1}^m \Big|\int_0^t \big(\widetilde f(s,G,\widetilde Z) - \widetilde f(s,G, Z^{(n)})\big)_{ik}\diff (Y_s)_k\Big|^2\Big]\\
&\leq \sum_{i=1}^{d} m \sum_{k=1}^m \E\Big[ \sup_{0\leq s< \nu_j}\Big|\int_0^t \big(\widetilde f(s,G,\widetilde Z) - \widetilde f(s,G, Z^{(n)})\big)_{ik}\diff (Y_s)_k\Big|^2\Big]\\
&\leq 4dm^2 \E\Big[ \sup_{0\leq s < \nu_j}\|f(s,G,\widetilde Z) - f(s,G, Z^{(n)})\|^2 V_{\nu_j-}\Big]\\
&\leq 4dm^2 j j^2 \E\Big[ \sup_{0\leq s< \nu_j} \|\widetilde Z_s - Z_s^{(n)}\|^2\Big] \to 0 \text{ as } n\to\infty.\\
\end{aligned}
\end{equation*}
Hence, recalling the definition of $Z^{(n)}$, we observe that $\widetilde Z$ is a solution of SDE \eqref{SDE} on $[0,T]$. Since $Z$ and $\widetilde Z$ are indistinguishable, we get that $ Z= \Psi (H,G, Y)$ is a solution of SDE \eqref{SDE}. This closes the proof of Theorem \ref{Main}.
\end{proof}

As an application of our result we illustrate how Theorem \ref{Main} can be used to prove in a straightforward way the existence of a regular conditional distribution. While the classical proof uses the measurability of the marginal distributions of the solution process, our proof is based on the following well-known result.

\begin{theorem}[\text{\cite[p.~185, Theorem~8.37]{klenke2014}}]\label{KlenkeRegCondDist}
Let $\mathcal{A}\subset \mathcal{F}$ be a sub-$\sigma$-algebra. Let $Z$ be a random variable with values in a Polish space $(E,\mathcal{E})$. Then there exists a regular conditional distribution $\kappa_{Z,\mathcal{A}}$ of $Z$ given $\mathcal{A}$.
\end{theorem}

\begin{corollary}\label{RegCondDist}
Let Assumption \ref{assf} hold. Let $t_0\in[0,T]$. Then the regular conditional distribution of $(X_s)_{s\in[t_0,T]}$ given $X_{t_0}$ exists, i.e. a stochastic kernel $\kappa_{(X_s)_{s\in[t_0,T]},\sigma(X_{t_0})}$ from $(\Omega,\mathcal{F})$ to $(\D_d,\B(\D_d))$ exists such that for all $A\in\mathcal{F}$ and all $B\in\B(\D_d)$ it holds that
\begin{equation*}
\begin{aligned}\label{RegCondDistDef}
\int_A \mathds{1}_{B}((X_s)_{s\in[t_0,T]}) \diff \P = \int_A \kappa_{(X_s)_{s\in[t_0,T]},\sigma(X_{t_0})}(\cdot, B)\diff \P.
\end{aligned}    
\end{equation*}
\end{corollary}

\begin{proof}
The measurability of the stochastic processes $H$, $G$, and $Y$ implies that the mapping
\begin{equation*}
\begin{aligned}
\omega \mapsto X(\omega) =\Psi(H(\omega),G(\omega),Y(\omega))
\end{aligned}
\end{equation*}
is a measurable mapping with values in $(\D_d,\mathcal{B}(\D_d))$ as a composition of measurable mappings. Hence for $t_0\in[0,T]$ obviously also the mapping 
\begin{equation*}
\begin{aligned}
\omega \mapsto \Psi(H(\omega),G(\omega),Y(\omega))\mathds{1}_{s\in[t_0,T]}
\end{aligned}
\end{equation*}
is measurable. 
Next, we use Theorem \ref{KlenkeRegCondDist} with $Z = \Psi(H,G,Y)\mathds{1}_{s\in[t_0,T]}$ and $\mathcal{A}=\sigma(X_{t_0})$ to obtain that there exists a regular conditional distribution $\kappa_{\Psi(H,G,Y)\mathds{1}_{s\in[t_0,T]},\sigma(X_{t_0})}$.  
Hence, $\P^{(X_s)_{s\in[t_0,T]}|X_{t_0}= x}$ exists and is given for all $B\in\B(\D_d)$ by
\begin{equation*}
\begin{aligned}
\P^{(X_s)_{s\in[t_0,T]}|X_{t_0}= x}(B) = \kappa_{\Psi(H,G,Y)\mathds{1}_{s\in[t_0,T]},\sigma(X_{t_0})} (X_{t_0}^{-1}(x),B).
\end{aligned}
\end{equation*}
\end{proof}

\section{Application to Malliavin calculus for pure-jump L\'evy processes with drift}

In this section we derive a universal expression for the Malliavin derivative for solutions of SDE \eqref{SDE} in the case that $Y$ is a pure-jump Lévy process with drift.

Let $\nu$ be the L\'evy-measure of $Y$ and let $N$ be its Poisson random measure, respectively $\tilde{N}$ its compensated Poisson random measure. We assume that $\mathcal{F}$ is the completion of $\sigma(Y)$. We can represent $Y$ almost surely by
\begin{align*}
Y_t=\gamma t+\int_{]0,t]\times\{|x|> 1\}}xN(ds,dx)+\int_{{]0,t]}\times\{|x|\leq 1\}}x\tilde{N}(ds,dx),\quad t\in[0,T],
\end{align*}
see, e.g., \cite{applebaum} or \cite{satou}.

Malliavin calculus can be used for determining existence of densities or existence of solutions to SDEs, backward stochastic differential equations, and stochastic partial integro-differential equations. For references in these directions, see, e.g., \cite{DelongImk,dinun,geiss2018monotonic,NualartSchoutens,Petrou}, of which \cite{Petrou,geiss2018monotonic} deal with the L\'evy setting we consider here.

There are several ways to define a Malliavin derivative in the L\'evy setting. We follow the approach used in \cite{suv}. There it is defined via chaotic decompositions of random variables. However, in \cite{suv2}, they show that on a canonical probability space carrying a L\'evy process, their definition equals a difference between a 'shifted' random variable and the original one. The connection to functional representations for random variables on an arbitrary probability space carrying a L\'evy process is presented in \cite{Steinicke2016}. In that spirit, we define

\begin{definition}
Let $Z\in L^2(\Omega,\mathcal{F},\mathbb{P})$, with functional representation $\psi_Z\colon \mathbb{D}_1\to\mathbb{R}$ measurable, $Z=\psi_Z(Y)$. Let for $r\in[0,T]$, $v\in \mathbb{R}$,
\begin{align*}
D_{r,v}Z := \psi_Z(Y+v \mathds{1}_{[r,T]})-\psi_Z(Y),
\end{align*}
as equivalence class in $L^2(\Omega\times[0,T]\times\mathbb{R},\mathcal{F}\otimes\mathcal{B}([0,T]\times\mathbb{R}),\mathbb{P}\otimes\lambda\otimes\nu)$,
whenever
\begin{equation}\label{CondMalDer}
\mathbb{E}\int_0^T\int_{\mathbb{R}}|D_{r,v}Z|^2\nu(dv)dr<\infty.
\end{equation}
The space of all $Z$ satisfying \eqref{CondMalDer} is denoted by $\mathcal{D}_{1,2}$.
\end{definition}

Our object of interest is the solution process of the SDE
\begin{align*}
X_t=H_t+\int_{(0,t]}g(s,G,X)dY_s,
\end{align*}
of which we want to investigate the Mallivain derivative $D_{r,v}X_t$. By our main theorem, we see that we may express the solution $X$ as $X=\Psi_g(H,G,Y)$. In addition we have $\mathbb{P}$-a.s.~representations for the random variables $H_t, G_t$ for all $t\in [0,T]$ through $H_t=\phi_{H_t}(Y), G_t=\phi_{G_t}(Y)$ for measurable functions $\phi_{H_t}, \phi_{G_t}\colon \mathbb{D}_1\to\mathbb{R}$, see \cite{Steinicke2016}. Since $H$ and $G$ are c\`adl\`ag processes, they are determined by their random variables evaluated at rational time points. Hence, we get a $\mathbb{P}$-a.s.~representation of the whole processes $H, G$ by measurable functionals $\phi_H,\phi_G\colon\mathbb{D}_1\to\mathbb{D}_1$ such that $H=\phi_H(Y), G=\phi_G(Y)$, where $(h\mapsto h_t)\circ\phi_H=\phi_{H_t}, (h\mapsto h_t)\circ\phi_G=\phi_{G_t}$ for all $t\in [0,T]$. If for all such $t$ the random variables $H_t$ and $G_t$ are in $\mathcal{D}_{1,2}$, the derivative $D_{r,v}H$ can be defined by
\begin{align*}
D_{r,v}H=\phi_H(Y+v\mathds{1}_{[r,T]})-\phi_H(Y);
\end{align*}
in the same way we can define it for $D_{r,v}G$.

With these representations, we can explicitly express the Malliavin derivative of $X_t$ in dependence only of $H, G$ and the functional $\Psi_g$.

\begin{theorem}\label{ThmMalCalc}
Let $X$ be the solution to \eqref{SDE} where $H, G$ are such that for all $t\in [0,T]$, $H_t, G_t\in \mathcal{D}_{1,2}$ and assume $X_t\in L^2(\Omega,\mathcal{F},\mathbb{P})$. Then for all $t\in [0,T]$,
\begin{align*}
D_{r,v}X_t=\Psi_g(H+D_{r,v}H,G+D_{r,v}G,Y+v\mathds{1}_{[r,T]})_t-\Psi_g(H,G,Y)_t,
\end{align*}
whenever it is square integrable w.r.t. $\mathbb{P}\otimes\lambda\otimes\nu.$
\end{theorem}

\begin{proof}
Since 
\begin{align*}
H=\phi_H(Y),\quad G=\phi_G(Y), \quad X=\Psi_g(H,G,Y),
\end{align*}
it follows that
\begin{align*}
X=\Psi_g(\phi_H(Y),\phi_G(Y),Y).
\end{align*}
Hence, 
\begin{align*}
&D_{r,v} X_t=D_{r,v} \Psi_g(\phi_H(Y),\phi_G(Y),Y)_t\\
&=\Psi_g(\phi_H(Y+v\mathds{1}_{[r,T]}),\phi_G(Y+v\mathds{1}_{[r,T]}),Y+v\mathds{1}_{[r,T]})_t-\Psi_g(\phi_H(Y),\phi_G(Y),Y)_t\\
&=\Psi_g(H+D_{r,v}H,G+D_{r,v}G,Y+v\mathds{1}_{[r,T]})_t-\Psi_g(H,G,Y)_t.
\end{align*}
\end{proof}

Finally, we give an illustrative example of the usage of our functional representation for calculating a Malliavin derivative, using Theorem \ref{ThmMalCalc}.

\begin{example}
We consider the SDE
\begin{equation}
X_t = \xi + \int_0^t g(X_{s-})\diff Y_s,
\end{equation}
where $\xi\in\R$, $g$ is a Lipschitz continuous function, and $Y$ is a pure-jump Lévy process with drift with compensator $\nu$ satisfying $\int_{\R}\textcolor{blue}{|x|^2}\nu(dx)<\infty$.
Then $X_t$ is Malliavin differentiable for all $t\in[0,T]$ by \cite[Theorem 17.4]{dinun}. Further Assumption \ref{assf} is satisfied and by Theorem \ref{ThmMalCalc} we know that 
\begin{equation}
D_{r,v} X_t = \Psi_g(\xi, Y+ v\mathds{1}_{t\in[r,T]})_t - \Psi_g (\xi,Y)_{t},
\end{equation}
where $\Psi$ does not depend on an additional process $G$ since $g$ does not.  
By the universality of our functional representation we know
$\Psi_g(\xi, Y+ v\mathds{1}_{t\in[r,T]})= \widetilde X$, which satisfies the SDE
\begin{equation}
\widetilde X_t = \xi + \int_0^t g(\widetilde X_{s-})\diff (Y_s + v\mathds{1}_{s\in[r,T]}).
\end{equation}
This implies
\begin{equation}
\begin{aligned}
\widetilde X_t = 
\begin{cases}
X_t, & t\in[0,r),\\
X_r+ g(X_{r-})v + \int_r^t g(\widetilde X_{s-})\diff Y_s, & t\in[r,T].
\end{cases}
\end{aligned}
\end{equation}
Hence, it holds 
\begin{equation}
D_{r,v}X_t = \Big(\int_r^t( g(\widetilde X_{s-})- g(X_{s-}))\diff Y_s + g(X_{r-})v\Big)\mathds{1}_{t\in[r,T]}.
\end{equation}
\end{example}

\section*{Acknowledgements}
The authors thank Steffen Dereich, the referee of a previous -- significantly different -- version of this paper as well as the referee of the current version of the paper for useful comments.

\vspace{2em}
\centerline{\underline{\hspace*{16cm}}}

\noindent Pawe{\l} Przyby{\l}owicz  \\
Faculty of Applied Mathematics, AGH University of Science and Technology, Al.~Mickiewicza 30, 30-059 Krakow, Poland\\
pprzybyl@agh.edu.pl\\

\noindent Verena Schwarz \Letter \\
Department of Statistics, University of Klagenfurt, Universit\"atsstra\ss{}e 65-67, 9020 Klagenfurt, Austria\\
verena.schwarz@aau.at\\

\noindent Alexander Steinicke \\
Department of Mathematics and Information Technology, Montanuniversitaet Leoben, Peter-Tunner-Straße 25/I, 8700 Leoben, Austria\\
alexander.steinicke@unileoben.ac.at\\

\noindent Michaela Sz\"olgyenyi \\
Department of Statistics, University of Klagenfurt, Universit\"atsstra\ss{}e 65-67, 9020 Klagenfurt, Austria\\
michaela.szoelgyenyi@aau.at\\

\end{document}